\documentclass[bsl,reqno]{asl}

\usepackage{savesym}
\usepackage{tikz}
\usetikzlibrary{matrix,arrows}
\usepackage{txfonts}
\usepackage{tikz-cd}
\usepackage{stmaryrd}
\usepackage{graphicx}
\usepackage{enumitem}
\usepackage{proof}

\newcommand{\tuple}[1]{\ensuremath{\langle{#1}\rangle}}
\newcommand{\eq}{\approx}
\newcommand{\N}{\ensuremath{\mathbb{N}}}

\newcommand{\defiff}{:\Longleftrightarrow}

\newcommand{\ops}{\ensuremath{\mathop{\star}}}
\newcommand{\lang}{\mathcal{L}}
\newcommand{\langs}{\mathcal{L}_s}
\newcommand{\model}[1]{{\mathfrak{{#1}}}}
\newcommand{\semvalue}[1]{\ensuremath{\left\|{#1}\right\|}}
\newcommand{\all}{\ensuremath{\forall}}
\newcommand{\exi}{\ensuremath{\exists}}
\newcommand{\All}[1]{\ensuremath{(\all{#1})}}
\newcommand{\Exi}[1]{\ensuremath{(\exi{#1})}}
\newcommand{\K}{\cls{K}}
\newcommand{\V}{\cls{V}}
\newcommand{\mK}{\cls{mK}}
\newcommand{\mV}{\cls{mV}}
\newcommand{\f}{\ensuremath{\varphi}}
\newcommand{\p}{\ensuremath{\psi}}
\newcommand{\x}{\ensuremath{\chi}}
\newcommand{\ep}{\ensuremath{\varepsilon}}

\newcommand{\pd}{\cdot}
\newcommand{\zr}{{\rm f}}
\newcommand{\ut}{{\rm e}}
\newcommand{\mt}{\land}
\newcommand{\jn}{\lor}

\newcommand{\mathrmL}{{\mathchoice{\mbox{\rm\L}}{\mbox{\rm\L}}{\mbox{\rm\scriptsize\L}}{\mbox{\rm\tiny\L}}}}

\renewcommand{\a}{\ensuremath{\alpha}}
\renewcommand{\b}{\ensuremath{\beta}}

\newcommand{\bo}{\ensuremath{\Box}}
\newcommand{\di}{\Diamond}

\newcommand{\mfram}[1]{\mathcal{#1}}

\newcommand{\De}{\mathrm{\Delta}}
\newcommand{\Ga}{\mathrm{\Gamma}}

\newcommand{\Si}{\mathrm{\Sigma}}
\newcommand{\fosc}[1]{\vDash^{\forall}_{#1}}
\newcommand{\mdl}[1]{\vDash_{#1}}
\newcommand{\nmdl}[1]{\not\vDash_{#1}}
\newcommand{\der}[1]{\vdash_{_\lgc{#1}}}
\newcommand{\lgc}[1]{\mathrm{#1}}
\newcommand{\alg}[1]{\mathbf{#1}}
\newcommand{\cls}[1]{\mathsf{#1}}

\newcommand*{\pfa}[1]{\mbox{\footnotesize $#1$}} 	 

\newcommand{\mfml}{\ensuremath{{\rm Fm}_{\bo}}}
\newcommand{\ofml}{\ensuremath{{\rm Fm}_{\forall}^{1}}}
\newcommand{\ofmls}{\ensuremath{{\rm Fm}_{\forall}^{1+}}}

\newcommand*{\md}{{\rm md}}
\newcommand*{\height}{{\rm ht}}
\newcommand*{\fFLe}{\lgc{\all1FL_e}}

\newcommand*{\seq}{{\vphantom{A}\Rightarrow{\vphantom{A}}}}
\newcommand*{\rseq}{\Rightarrow}
\newcommand{\idr}{(\textsc{id})}

\newcommand{\flr}{(\zr\!\rseq)}
\newcommand{\frr}{(\rseq\!\zr)}
\newcommand{\tlr}{(\ut\!\rseq)}
\newcommand{\trr}{(\rseq\!\ut)}
\newcommand{\olr}{(\jn\!\rseq)}
\newcommand{\orr}{(\rseq\!\jn)}
\newcommand{\alr}{({\mt\!\rseq})}
\newcommand{\arr}{({\rseq\!\mt})}
\newcommand{\falr}{({\all\!\rseq})}
\newcommand{\farr}{(\rseq\!\all)}
\newcommand{\elr}{(\exi\!\rseq)}
\newcommand{\err}{(\rseq\!\exi)}
\newcommand{\pdlr}{(\pd\!\rseq)}
\newcommand{\pdrr}{(\rseq\!\pd)}
\newcommand{\ilr}{(\to\rseq)}
\newcommand{\irr}{(\rseq\to)}

\newcommand{\usecolor}[2]{{\color{#1} #2}}
\newcommand{\petr}[1]{\usecolor{red}{#1}}

\title[One-variable fragments of first-order logics]
{One-variable fragments of first-order logics}


\author{Petr Cintula}
\revauthor{Cintula, Petr}
\address{Institute of Computer Science\\
Czech Academy of Sciences\\
Prague, Czech Republic}
\email{cintula@cs.cas.cz}

\author{George Metcalfe}
\revauthor{Metcalfe, George}
\address{Mathematical Institute\\
University of Bern\\ 
Bern, Switzerland}
\email{george.metcalfe@unibe.ch}

\author{Naomi Tokuda}
\revauthor{Metcalfe, George}
\address{Mathematical Institute\\
University of Bern\\
Bern, Switzerland}
\email{naomi.tokuda@unibe.ch}

\thanks{The first author was supported by RVO 67985807 and Czech Science Foundation grant GA22-01137S, and the second two authors by Swiss National Science Foundation grant 200021\textunderscore 215157. This project has also received funding from the European Union’s Horizon 2020 research and innovation programme under the Marie Sk{\l}odowska-Curie grant agreement No 101007627. }

\theoremstyle{definition}
\newtheorem{theorem}{Theorem}

\newtheorem{lemma}[theorem]{Lemma}

\newtheorem*{claim*}{Claim}

\newtheorem{corollary}[theorem]{Corollary}
\newtheorem{proposition}[theorem]{Proposition}
\newtheorem{example}[theorem]{Example}

\newtheorem{named-remark}[theorem]{}
\renewcommand{\labelenumi}{(\roman{enumi})}
\numberwithin{theorem}{section}

\begin{document}

\begin{abstract}
The one-variable fragment of a first-order logic may be viewed as an ``S5-like'' modal logic, where the universal and existential quantifiers are replaced by box and diamond modalities, respectively. Axiomatizations of these modal logics have been obtained for special cases --- notably, the modal counterparts $\lgc{S5}$ and $\lgc{MIPC}$ of the one-variable fragments of first-order classical logic and intuitionistic logic --- but a general approach, extending beyond first-order intermediate logics, has been lacking. To this end, a sufficient criterion is given in this paper for the one-variable fragment of a semantically-defined first-order logic --- spanning families of intermediate, substructural, many-valued, and modal logics --- to admit a natural axiomatization. More precisely, such an axiomatization is obtained for the one-variable fragment of any first-order logic based on a variety of algebraic structures with a lattice reduct that has the superamalgamation property, building on a generalized version of a functional representation theorem  for monadic Heyting algebras due to Bezhanishvili and Harding.  An alternative proof-theoretic strategy for obtaining such axiomatization results is also developed for first-order substructural logics that have a cut-free sequent calculus and admit a certain interpolation property.
\end{abstract}


\maketitle


\vspace{-3ex} \petr{}

\section{Introduction}\label{s:introduction}

The one-variable fragment of any standard first-order logic --- intermediate, substructural, many-valued, modal,  or otherwise --- consists of consequences in the logic constructed using one distinguished variable $x$, unary relation symbols, propositional connectives, and the quantifiers $\All{x}$ and $\Exi{x}$. Such a fragment may be conveniently reformulated as a propositional modal logic by replacing occurrences of an atom $P(x)$ with a propositional variable $p$, and occurrences of $\All{x}$ and $\Exi{x}$ with $\bo$ and $\di$, respectively. Typically, this modal logic is algebraizable --- that is, it enjoys soundness and completeness with respect to some suitable class of algebraic structures -- and hence, unlike the full first-order logic, can be studied using the tools of universal algebra.

Any standard semantics for a first-order logic, where quantifiers range over domains of models, yields a relational semantics for the one-variable fragment. On the other hand, a Hilbert-style axiomatization does not (at least directly) yield an axiomatization for the fragment, since a derivation of a one-variable formula may involve additional variables. Axiomatizations are well known for the modal counterparts $\lgc{S5}$~\cite{Hal55} and $\lgc{MIPC}$~\cite{MV57,Bul66} of the one-variable fragments of first-order classical logic and intuitionistic logic, respectively, and similar results have been obtained for modal counterparts of one-variable fragments of other first-order intermediate logics~\cite{OS88,Suz89,Suz90,Bez98,BH02,CR15,CMRR17,CMRT22} and many-valued logics~\cite{Rut59,dNG04,CCVR20,MT20}. However, a general approach to axiomatizing one-variable fragments of first-order logics has, until now, been lacking.\footnote{A precursor to this paper, reporting preliminary results restricted to a smaller class of logics, was published in the proceedings of AiML~2022~\cite{CMT22}.}

In this paper, we address the aforementioned axiomatization problem for a broad family of semantically-defined first-order logics. First, in Section~\ref{s:one-variable}, we introduce (one-variable) first-order logics based on models defined over classes of {\em $\lang$-lattices}: structures for an algebraic signature $\lang$ that have a lattice reduct.  In particular, first-order intermediate and substructural logics can be defined over classes of Heyting algebras and $\cls{FL_e}$-algebras, respectively. For the sake of generality (e.g., when $\lang$-lattices are just lattices), consequence is defined over equations between first-order formulas; however, this often --- in particular, for intermediate and substructural logics --- corresponds to the usual notion of consequence between formulas. 

In Section~\ref{s:algebraic}, we introduce potential axiomatizations for consequence in the modal counterparts of the one-variable fragments of these  semantically-defined first-order logics. We define an {\em m-$\lang$-lattice} to be an $\lang$-lattice expanded with modalities $\bo$ and $\di$ satisfying certain equations familiar from modal logic, and given any class $\K$ of $\lang$-lattices, let $\mK$ denote the class of m-$\lang$-lattices with an $\lang$-lattice reduct in $\K$. For example, if $\K$ is a variety of Heyting algebras, then $\mK$ is a variety of monadic Heyting algebras in the sense of~\cite{MV57}. We then show that m-$\lang$-lattices are in one-to-one correspondence with $\lang$-lattices equipped with a subalgebra satisfying a relative completeness condition,  generalizing previous results in the literature (see, e.g.,~\cite{Bez98,Tuy21}). We also show that if $\K$ is any class  of $\lang$-lattices closed under taking subalgebras and direct powers (in particular, any variety), then consequence in the one-variable fragment of the first-order logic defined over $\K$ corresponds to consequence in the {\em functional} members of $\mK$: m-$\lang$-lattices consisting of functions from a set $W$ to an $\lang$-lattice $\alg{A}\in\K$.

In Section~\ref{s:functional}, we close the circle, obtaining an axiomatization of consequence in the one-variable fragment of any first-order logic defined over a variety of $\lang$-lattices that has the {\em superamalgamation property}: a well-studied algebraic property equivalent in some cases to Craig interpolation for the associated logic. That is, we show that for such a variety $\V$, every member of $\mV$ is functional --- generalizing Bezhanishvili and Harding's representation theorem for monadic Heyting algebras~\cite{BH02} --- and hence that the defining equations for $\mV$ provide the desired axiomatization. As a consequence, we obtain axiomatizations of the one-variable fragments of a broad range of first-order logics, including the seven consistent first-order intermediate logics admitting Craig interpolation, first-order extensions of substructural logics such as $\cls{FL_e}$, $\cls{FL_{ew}}$, and $\cls{FL_{ec}}$, a first-order lattice logic, and a first-order version of the modal logic $\lgc{K}$.
 
In Section~\ref{s:prooftheory}, we present an alternative proof-theoretic strategy for establishing completeness of an axiomatization for the one-variable fragment of a first-order logic, the key idea being to show that additional variables can be eliminated from derivations of one-variable formulas in a suitable sequent calculus. As a concrete example, we obtain a new completeness proof for the one-variable fragment of the first-order version of the substructural logic $\lgc{FL_e}$ by establishing an interpolation property for derivations in a cut-free sequent calculus. We then explain how the proof generalizes to a family of first-order substructural logics, including  $\lgc{FL_{ew}}$,  $\lgc{FL_{ec}}$, and $\lgc{FL_{ewc}}$ (intuitionistic logic). Finally, in Section~\ref{s:concluding}, we discuss the limitations of the methods described in the paper and potential extensions to broader families of first-order logics.


\section{A family of first-order logics}\label{s:one-variable}

Let $\lang$ be any algebraic signature, and let $\lang_n$ denote the set of operation symbols of $\lang$ of arity $n\in\N$. We will assume throughout this paper that $\lang_2$ contains distinct symbols $\mt$ and $\jn$, referring to such a signature as {\em lattice-oriented}.

We call an algebraic structure $\alg{A} = \tuple{A,\{\ops^{\alg{A}}\mid n\in\N,\,\ops\in\lang_n\}}$ an {\em $\lang$-lattice} if $\ops^{\alg{A}}$ is an $n$-ary operation on $A$ for each $\ops\in\lang_n$ ($n\in\N$), and $\tuple{A,\mt^{\alg{A}},\jn^{\alg{A}}}$ is a lattice with respect to the induced order $x\leq^{\alg{A}}y \defiff x \mt^{\alg{A}} y = x$. As usual, superscripts will be omitted when these are clear from the context.

\begin{example}\label{e:substructural}
Let $\lang_s$ be the lattice-oriented signature with binary operation symbols $\mt$, $\jn$, $\pd$, and $\to$, and constant symbols $\zr$ and $\ut$. An {\em $\cls{FL_e}$-algebra} --- also referred to as a {\em commutative pointed residuated lattice} --- is an $\lang_s$-lattice $\alg{A}=\tuple{A,\mt,\jn,\pd,\to,\zr,\ut}$ such that $\tuple{A,\cdot,\ut}$ is a commutative monoid and $\to$ is the residuum of $\pd$, that is, $a\pd b\leq c\iff a\leq b\to c$, for all $a,b,c\in~A$. The class of $\cls{FL_e}$-algebras forms a variety $\cls{FL_e}$ that provides algebraic semantics for the full Lambek calculus with exchange $\lgc{FL_e}$ --- also known as multiplicative additive intuitionistic linear logic without additive constants  (see, e.g.,~\cite{GJKO07,MPT23}). Algebraic semantics for other well-known substructural logics are provided by various subvarieties of $\cls{FL_e}$; in particular,
\begin{enumerate}[font=\upshape, label={$\bullet$}, topsep=2mm,  itemsep=2mm]

\item the full Lambek calculus with exchange and weakening $\lgc{FL_{ew}}$, and full Lambek calculus with exchange and contraction $\lgc{FL_{ec}}$, correspond to  the varieties $\cls{FL_{ew}}$ and $\cls{FL_{ec}}$ of $\cls{FL_e}$-algebras satisfying the equations $\zr\leq x\leq\ut$, and $x\leq x\pd x$, respectively;

\item intuitionistic logic $\lgc{IL}$ corresponds to the variety $\cls{HA}$ of Heyting algebras, term-equivalent to $\cls{FL_{ewc}}=\cls{FL_{ew}}\cap\cls{FL_{ec}}$  (just identify $\pd$ and $\mt$);

\item classical logic $\lgc{CL}$ and G{\"o}del logic $\lgc{G}$ correspond to the varieties $\cls{BA}$ of Boolean algebras, and  $\cls{GA}$ of G{\"o}del algebras, axiomatized relative to $\cls{HA}$ by the equations $(x\to\zr)\to\zr\eq x$ and $(x\to y)\jn(y\to x)\eq\ut$, respectively;

\item  {\L}ukasiewicz logic $\lgc{\mathrmL}$ corresponds to the variety $\cls{MV}$ of MV-algebras, term-equivalent to the variety of $\cls{FL_{ew}}$-algebras satisfying $(x\to y)\to y\eq x\jn y$.

\end{enumerate}
\end{example}

Full first-order logics can be defined over an arbitrary predicate language with formulas built using propositional connectives in the algebraic signature~$\lang$  (see,~e.g.,~\cite[Section~7.1]{CN21}). However, for the purposes of this paper it suffices to restrict our attention to the one-variable setting with a fixed (generic) predicate language. Let $\ofml(\lang)$ denote the set of {\em one-variable $\lang$-formulas} $\f,\p,\dots$, built  inductively as usual from a countably infinite set of unary predicates $\{P_i\}_{i\in\N}$, a distinguished variable $x$, connectives in $\lang$, and quantifiers $\all,\exi$. We also call an ordered pair of one-variable $\lang$-formulas $\f,\p\in\ofml(\lang)$, written $\f\eq\p$, an {\em $\ofml(\lang)$-equation}, and let $\f\leq\p$ denote $\f\mt\p\eq\f$.\footnote{Let us emphasize that an $\ofml(\lang)$-equation $\f\eq\p$ is a primitive syntactic object that relates two formulas and not terms. In some settings (e.g., first-order substructural logics), $\f\eq\p$ can be replaced by a formula such as $\f \leftrightarrow \p$ and semantical consequence can be defined between formulas, but this is not always the case.}

Now let $\alg{A}$ be any $\lang$-lattice, let $S$ be a non-empty set, and let $\mfram{I}(P_i)$ be a map from $S$ to $A$ for each $i\in\N$, writing $u\mapsto f(u)$ to denote a map assigning to each $u\in S$ some $f(u)\in A$. We call the ordered pair $\model{S}=\tuple{S,\mfram{I}}$ an {\em $\alg{A}$-structure} if the following inductively defined partial map $\semvalue{\cdot}^\model{S}\colon\ofml(\lang)\to A^S$ is total:
\begin{align*}
\semvalue{P_i(x)}^\model{S} &= \mathcal{I}(P_i)\quad && i\in\N\\
\semvalue{\ops(\f_1,\dots,\f_n)}^\model{S}
&=u\mapsto{\ops}^\alg{A}\big(\semvalue{\f_1}^\model{S}(u)\,\dots ,\,\semvalue{\f_n}^\model{S}(u)\big)\quad && n\in\N,\ops\in\lang_n\\
\semvalue{\All{x}\f}^\model{S} &= u\mapsto\bigwedge\big\{\semvalue{\f}^{\model{S}}(v) \mid v\in S\big\}\\
\semvalue{\Exi{x}\f}^\model{S} &= u\mapsto\bigvee\big\{\semvalue{\f}^{\model{S}}(v) \mid v\in S\big\}.
\end{align*}
If $\alg{A}$ is {\em complete} --- that is, $\bigwedge X$ and $\bigvee X$ exist in $A$, for all $X \subseteq A$ --- then $\model{S}=\tuple{S,\mfram{I}}$ is always an $\alg{A}$-structure; otherwise, whether or not the partial map $\semvalue{\cdot}^\model{S}$ is total depends on $\mfram{I}$. E.g., for $\alg{A}=\tuple{\mathbb{N},\min,\max}$ and $S=\N$, if $\mfram{I}(P_0)(n)\coloneqq n$, for all $n\in\N$, then $\semvalue{\Exi{x}P_0(x)}^\model{S}$ is undefined, but if $\mfram{I}(P_i)(n)\le K$ for all $i\in\N$ and $n\in S$, for some fixed $K\in\N$, then  $\model{S}$ is an $\alg{A}$-structure.

We say that an $\ofml(\lang)$-equation $\f\eq\p$ is {\em valid} in an $\alg{A}$-structure $\model{S}$, and write $\model{S}\models\f\eq\p$, if $\semvalue{\f}^\model{S}=\semvalue{\p}^\model{S}$. More generally, consider any class of $\lang$-lattices $\K$. We say that an $\ofml(\lang)$-equation $\f\eq\p$ is a {\em (sentential) semantical consequence}  of a set of $\ofml(\lang)$-equations $T$ in $\K$, and write $T\fosc{\K}\f\eq\p$, if for any $\alg{A}\in\K$ and $\alg{A}$-structure $\model{S}$,
\begin{align*}
\model{S}\models\f'\eq\p'\text{, for all }\f'\eq\p' \in T\enspace\Longrightarrow\enspace\model{S}\models\f\eq\p.
\end{align*}
In certain cases, we can restrict attention to the complete members of $\K$. Let us say that $\K$ {\em admits regular completions} if, for any $\alg{A}\in\K$, there exists an $\lang$-lattice embedding $h$ of $\alg{A}$ into a complete member $\alg{B}$ of $\K$ that preserves all existing meets and joins, noting that for any $\alg{A}$-structure $\model{S}=\tuple{S,\mfram{I}}$, the $\alg{B}$-structure $\model{S}^h=\tuple{S,\mfram{I}^h}$, with $\mfram{I}^h(P_i)\coloneqq h\circ\mfram{I}(P_i)$ for each $i\in I$, satisfies $\semvalue{\f}^{\model{S}^h}=h\circ\semvalue{\f}^{\model{S}}$ for each  $\varphi\in\ofml(\lang)$.  Clearly, semantical consequence in such a class $\K$ coincides with semantical consequence in the class of complete members of $\K$.

\begin{example}
A sufficient, but by no means necessary, condition for a class of $\lang$-lattices to admit regular completions is closure under MacNeille completions (see, e.g.,~\cite{Har08}). This is the case in particular for $\cls{BA}$ and $\cls{HA}$; indeed, they are the only non-trivial varieties of Heyting algebras that have this property~\cite{BH04}. A broad family of varieties of $\cls{FL_e}$-algebras --- including $\cls{FL_e}$, $\cls{FL_{ew}}$, and $\cls{FL_{ec}}$ --- are also closed under MacNeille completions (see,~e.g.,~\cite{CGT12}), and for a still broader family --- including  $\cls{GA}$ --- this is true for the class of their subdirectly irreducible members~\cite{CGT11}. Note, however, that in some cases --- e.g.,  $\cls{MV}$~\cite{GP02} --- neither the variety nor the class of its subdirectly irreducible members admits regular completions.
\end{example}

Next, let us denote by $\mfml(\lang)$ the set of propositional formulas $\a,\b,\dots$ built inductively as usual from a countably infinite set of propositional variables $\{p_i\}_{i\in\N}$, connectives in $\lang$, and unary connectives $\bo$ and $\di$, and call an ordered pair of formulas $\a,\b\in\mfml(\lang)$, written $\a\eq\b$,  an {\em $\mfml(\lang)$-equation}. The (standard) translation functions $({-})^\ast$ and $({-})^\circ$ between $\ofml(\lang)$ and $\mfml(\lang)$ are defined inductively by
\begin{align*}
(P_i(x))^\ast			&= p_i						 & p_i^\circ		&= P_i(x) & i\in\N\\
(\ops(\f_1,\dots,\f_n))^\ast	&= \ops(\f^\ast_1,\dots,\f^\ast_n)	\ \   & (\ops(\a_1,\dots,\a_n))^\circ		&= \ops(\a^\circ_1,\dots,\a^\circ_n) & \ops\in\lang_n\\
(\All{x} \f)^\ast			&=\bo \f^\ast					& (\bo\a)^\circ	&= \All{x} \a^\circ \\
(\Exi{x} \f)^\ast			&=\di \f^\ast					& (\di\a)^\circ	&= \Exi{x} \a^\circ,
\end{align*}
and lift in the obvious way to (sets of) $\ofml(\lang)$-equations and $\mfml(\lang)$-equations. 

Clearly, $(\f^\ast)^\circ = \f$ for any $\f\in\ofml(\lang)$ and $(\a^\circ)^\ast = \a$ for any $\a\in\mfml(\lang)$, and we may therefore switch between first-order and modal notations as convenient. Indeed, to achieve our goal of axiomatizing consequence in the one-variable first-order logic based on a class of $\lang$-lattices $\K$, it suffices to find a (natural) axiomatization of a variety $\V$ of algebras in the signature of $\lang$ expanded with $\bo,\di$ such that $\fosc{\K}$ corresponds to equational consequence in~$\V$.  More precisely, let us call a homomorphism from the formula algebra with universe $\mfml(\lang)$ to $\alg{A}\in\V$ an {\em $\alg{A}$-evaluation}, and define for any set $\Si\cup\{\a\eq\b\}$ of $\mfml(\lang)$-equations,
\begin{align*}
\Si\mdl{\V} \a\eq\b\:\defiff \enspace
& \text{$f(\a)=f(\b)$, for every $\alg{A}\in\V$ and $\alg{A}$-evaluation $f$}\\
&   \text{satisfying $f(\a')=f(\b')$ for all $\a'\eq\b'\in\Si$.}
\end{align*}
Our goal is to provide a  (natural) axiomatization of a variety $\V$ such that for any set of $\ofml(\lang)$-equations $T\cup\{\f\eq\p\}$,
\begin{align*}
T\fosc{\K}\f\eq\p\:\iff\:T^\ast\mdl{\V}\f^\ast\eq\p^\ast.
\end{align*}

\begin{example}
If $\K$ is $\cls{BA}$, then $\fosc{\K}$ is consequence in the one-variable fragment of first-order classical logic, corresponding to $\lgc{S5}$, and $\V$ is the variety of monadic Boolean algebras defined in~\cite{Hal55}. If $\K$ is $\cls{HA}$, then $\fosc{\K}$ is consequence in the one-variable fragment of first-order intuitionistic logic, corresponding to $\lgc{MIPC}$, and $\V$ is the variety of monadic Heyting algebras defined in~\cite{MV57}. Analogous results have been obtained for first-order intermediate logics~\cite{OS88,Suz89,Suz90,Bez98,BH02,CR15,CMRR17,CMRT22}. In particular,  if $\K$ is $\cls{GA}$, then $\fosc{\K}$ is consequence in the one-variable fragment of the first-order logic of linear frames, and $\V$ is the variety of monadic Heyting algebras satisfying the prelinearity axiom  $(x\to y)\jn(y\to x)\eq\ut$~\cite{CMRT22}. However, if $\K$ is the class of totally ordered members of $\cls{GA}$, then $\fosc{\K}$ is consequence in the one-variable fragment of first-order G{\"o}del logic, the first-order logic of linear frames with a constant domain, and $\V$ is the variety of monadic G{\"o}del algebras, i.e., monadic Heyting algebras satisfying the prelinearity axiom  and the constant domain axiom $\bo(\bo x\jn y)\eq\bo x\jn\bo y$~\cite{CR15}. Similarly, if $\K$ is the class of totally ordered MV-algebras, then $\fosc{\K}$ is consequence in the one-variable fragment of first-order \L ukasiewicz logic, and $\V$ is the variety of monadic MV-algebras~\cite{Rut59}.
\end{example}


\section{An algebraic approach}\label{s:algebraic}

As our basic modal structures, let us define an \emph{m-lattice} to be any algebraic structure $\tuple{L,\mt,\jn,\bo,\di}$ with lattice reduct $\tuple{L,\mt,\jn}$ that satisfies the following equations:
\[
\begin{array}{r@{\quad}l@{\qquad\qquad}r@{\quad}l}
{\rm (L1_\bo)}	&\bo x \mt x \eq\bo x 		& {\rm (L1_\di)}	 & \di x \jn x \eq\di x\\
{\rm (L2_\bo)}	&\bo(x \mt y) \eq\bo x \mt\bo y	& {\rm (L2_\di)}	 &\di (x\jn y)\eq\di x \jn\di y\\
{\rm (L3_\bo)}	&\bo\di x \eq\di x			& {\rm (L3_\di)}	 &\di\bo x \eq\bo x. 
\end{array}
\]
Let $\a\le\b$ stand for $\a\mt\b\eq\a$. It is easily shown that every m-lattice also satisfies the following equations and quasi-equations:
\[
\begin{array}{r@{\quad}l@{\qquad\qquad}r@{\quad}l}
{\rm (L4_\bo)}	& \bo\bo x \eq\bo x 	& {\rm (L4_\di)} & \di\di x \eq\di x \\
{\rm (L5_\bo)}	& x \le y \,\Longrightarrow\,\bo x \le\bo y 	& {\rm (L5_\di)} 	& x \le y \,\Longrightarrow\,\di x \le\di y.
\end{array}
\]
Now let $\lang$ be any fixed lattice-oriented signature. We define an {\em m-$\lang$-lattice} to be any algebraic structure $\tuple{\alg{A},\bo,\di}$ such that $\alg{A}$ is an $\lang$-lattice, $\tuple{A,\mt,\jn,\bo,\di}$ is an m-lattice, and the following equation is satisfied for each $n\in\N$ and $\ops\in\lang_n$:
\[
\begin{array}{rl}
(\ops_\bo)	& \bo(\ops(\bo x_1,\dots,\bo x_{n}))\eq\ops(\bo x_1,\dots,\bo x_{n}).
\end{array}
\]
Using (\ops$_\bo$), {\rm (L3$_\bo$)}, and {\rm (L3$_\di$)}, it follows that $\tuple{\alg{A},\bo,\di}$ also satisfies for each $n\in\N$ and $\ops\in\lang_n$, the equation
\[
\begin{array}{rl}
(\ops_\di)	&  \di(\ops(\di x_1,\dots,\di x_{n}))\eq\ops(\di x_1,\dots,\di x_{n}).
\end{array}
\]
Finally, given a class $\K$ of $\lang$-lattices, let $\mK$ denote the class of m-$\lang$-lattices with an $\lang$-lattice reduct in $\K$. Note that if $\K$ is a variety, then so is $\mK$.

\begin{example}\label{e:mFLe}
It is straightforward to show that the notion of an m-$\langs$-lattice encompasses other algebraic structures considered in the literature. In particular, $\cls{mBA}$ and $\cls{mHA}$ are the varieties of monadic Heyting algebras~\cite{MV57} and monadic Boolean algebras~\cite{Hal55}, respectively. Moreover, if $\alg{A}$ is an $\cls{FL_e}$-algebra, then every m-$\langs$-lattice $\tuple{\alg{A},\bo,\di}$ satisfies the equations
\[
\begin{array}{r@{\quad}l@{\qquad\qquad}r@{\quad}l}
{\rm (L6_\bo)}	&\bo (x\to \bo y) \eq\di x \to\bo y	& {\rm (L6_\di)}	& \bo (\bo x\to y) \eq\bo x \to\bo y,
\end{array}
\]
and $\cls{mFL_e}$ is therefore the variety of monadic $\cls{FL_e}$-algebras introduced in~\cite{Tuy21}. Let us just check {\rm (L6$_\bo$)}, the proof for {\rm (L6$_\di$)} being very similar. Consider any $a,b\in A$. Since $a\le\di a$, by {\rm (L1$_\di$)}, also  $\di a\to\bo b\le a\to\bo b$. Hence, using {\rm (L3$_\bo$)}, {\rm ($\to_\bo$)}, and {\rm (L5$_\bo$)}, 
\begin{align*}
\di a\to\bo b = \bo \di a \to \bo b = \bo(\bo\di a \to\bo b) = \bo(\di a \to \bo b)\le \bo(a\to\bo b).
\end{align*}
Conversely, since $\bo (a\to\bo b) \le a\to\bo b$, by {\rm (L1$_\bo$)}, it follows by residuation that $a \le\bo (a\to\bo b)\to\bo b$ and hence, using {\rm (L5$_\di$)}, {\rm (L3$_\di$)}, and ($\to_\di$),
\begin{align*}
\di a \le\di(\bo (a\to\bo b)\to\bo b)=\bo (a\to\bo b)\to\bo b.
\end{align*}
By residuation again, $\bo (a\to\bo b) \le\di a \to\bo b$. 
\end{example}

\begin{example}\label{e:monadicvarieties}
The variety $\cls{mGA}$ corresponds to the one-variable fragment of Corsi's first-order logic of linear frames~\cite{CMRT22}, whereas the variety of monadic G{\"o}del algebras --- axiomatized relative to $\cls{mGA}$ by the constant domain axiom --- corresponds to the one-variable fragment of first-order G{\"o}del logic, the first-order logic of linear frames with a constant domain~\cite{CR15}. Note, however, that the variety  axiomatized relative to $\cls{mMV}$ by the constant domain axiom does not satisfy $\di x \pd \di x \eq \di (x \pd x)$ and therefore properly contains the variety of monadic MV-algebras studied in~\cite{Rut59,dNG04,CCVR20}. Consider, for example, the MV-algebra $\textbf{\L}_3 = \tuple{\{0,\frac12,1\},\mt,\jn,\pd,\to,0,1}$ (in the language of $\cls{FL_e}$-algebras) with the usual order, where $a\pd b:=\max(0,a+b-1)$ and $a\to b:=\min(1,1-a+b)$. Let $\bo 0 = \bo\frac12= \di 0 = 0$ and $\bo 1=\di\frac12=\di 1 =1$. Then $\tuple{\textbf{\L}_3,\bo,\di}\in\cls{mMV}$ satisfies the constant domain axiom, but $\di\frac12\pd\di\frac12= 1\pd 1=1\neq 0=\di 0=\di(\frac12\pd\frac12)$. 
\end{example}

We now provide a useful description of m-$\lang$-lattices that generalizes results in the literature for varieties such as monadic Heyting algebras~\cite{Bez98} and monadic $\cls{FL_e}$-algebras~\cite{Tuy21}.

\begin{lemma}\label{l:subalgebra}
Let $\tuple{\alg{A},\bo,\di}$ be any m-$\lang$-lattice. Then $\bo A\coloneqq\{\bo a\mid a\in A\}$ forms a subalgebra $\bo\alg{A}$  of $\alg{A}$, where $\bo A = \di A\coloneqq \{\di a\mid a\in A\}$ and for any $a\in A$,
\begin{align*}
\bo a = \max \{b\in\bo A\mid b\le a\} \quad\text{and}\quad\di a = \min\{b\in\bo A\mid a\le b\}.
\end{align*}
\end{lemma}
\begin{proof}
The fact that $\bo A$ forms a subalgebra of $\alg{A}$ follows directly using (\ops$_\bo$) for each operation symbol $\ops$ of $\lang$, and $\bo A =\di A$ follows from {\rm (L3$_\bo$)} and {\rm (L3$_\di$)}. Now consider any $a\in A$. If $b\in\bo A$ satisfies $b \le a$, then $b=\bo b\le \bo a$, by {\rm (L4$_\bo$)} and {\rm (L5$_\bo$)}. But $\bo a\le a$, by {\rm (L1$_\bo$)}, so $\bo a = \max \{b\in\bo A\mid b\le a\}$. Analogous reasoning yields $\di a = \min\{b\in\bo A\mid a\le b\}$.
\end{proof}

Let us call a sublattice $\alg{L}_0$ of a lattice $\alg{L}$ \emph{relatively complete} if for any $a\in L$, the set $\{b\in L_0\mid b\leq a\}$ contains a maximum and the set $\{b\in L_0\mid a\leq b\}$ contains a minimum. Equivalently, $\alg{L}_0$ is relatively complete if the inclusion map $f_0$ from $\tuple{L_0,\le}$ to $\tuple{L,\le}$ has left and right adjoints, that is, if there exist order-preserving maps $\bo\colon L\to L_0$ and $\di\colon L\to L_0$ such that for all $a\in L$ and $b\in L_0$,
\begin{align*}
f_0(b) \le a \iff b \le\bo a \quad\text{and}\quad a \le f_0(b) \iff\di a\le b.
\end{align*}
Let us also say that a subalgebra $\alg{A}_0$ of an $\lang$-lattice $\alg{A}$ is relatively complete if this property holds with respect to their lattice reducts. In particular, by Lemma~\ref{l:subalgebra}, the subalgebra $\bo\alg{A}$ of $\alg{A}$ is relatively complete for any m-$\lang$-lattice  $\tuple{\alg{A},\bo,\di}$. The following result establishes a converse.

\begin{lemma}\label{l:converse}
Let $\alg{A}_0$ be a relatively complete subalgebra of an $\lang$-lattice $\alg{A}$, and define $\bo_0 a \coloneqq \max \{b\in A_0\mid b\le a\}$ and $\di_0 a \coloneqq \min\{b\in A_0\mid a\le b\}$ for each $a\in A$. Then $\tuple{\alg{A},\bo_0,\di_0}$ is an m-$\lang$-lattice and $\bo_0 A =\di_0 A = A_0$.
\end{lemma}
\begin{proof}
It is straightforward to check that $\tuple{A,\mt,\jn,\bo_0,\di_0}$ is an m-lattice; for example, it satisfies {\rm (L2$_\bo$)}, since for any $a_1,a_2\in A$,
\begin{align*}
\bo_0 (a_1\mt a_2) 
& = \max \{b\in A_0\mid b\le a_1\mt a_2\}\\
& = \max \{b\in A_0\mid b\le a_1 \text{ and } b \le a_2\} \\
& = \max \{b\in A_0\mid b\le a_1\} \mt \max \{b\in A_0\mid b\le a_2\} \\
& = \bo_0 a_1 \mt \bo_0 a_2.
\end{align*}
Since $\alg{A}_0$  is a subalgebra of $\alg{A}$, clearly $\tuple{\alg{A},\bo_0,\di_0}$ also satisfies (\ops$_\bo$). Hence $\tuple{\alg{A},\bo_0,\di_0}$ is an m-$\lang$-lattice and $\bo_0 A =\di_0 A = A_0$.
\end{proof}

Lemmas~\ref{l:subalgebra} and~\ref{l:converse} together yield the following representation theorem for  m-$\lang$-lattices.

\begin{theorem}\label{t:correspondence}
Let $\K$ be any class of $\lang$-lattices. Then there exists a one-to-one correspondence between the members of $\mK$ and ordered pairs $\tuple{\alg{A},\alg{A}_0}$ such that $\alg{A}\in\K$ and $\alg{A}_0$ is a relatively complete subalgebra of $\alg{A}$, implemented by the maps $\tuple{\alg{A},\bo,\di}\mapsto\tuple{\alg{A},\bo\alg{A}}$ and $\tuple{\alg{A},\alg{A}_0}\mapsto\tuple{\alg{A},\bo_0,\di_0}$.
\end{theorem}

Next, given any  $\lang$-lattice  $\alg{A}$ and set $W$, let $\alg{A}^W$ be the $\lang$-lattice with universe $A^W$, where the operations are defined pointwise.

\begin{proposition}\label{p:standard}
Let $\alg{A}$ be an $\lang$-lattice, $W$ a set, and $\alg{B}$ a subalgebra of $\alg{A}^W$ such that for each $f\in B$, the elements $\bigwedge_{v\in W}f(v)$ and $\bigvee_{v\in W}f(v)$  exist in $\alg{A}$ and the following constant functions belong to $B$,
\begin{align*}
\bo f\colon W\to A;\:u\mapsto\bigwedge_{v\in W} f(v)
\quad\text{and}\quad\di f\colon W\to A;\:u\mapsto\bigvee_{v\in W}f(v).
\end{align*}
Then $\tuple{\alg{B},\bo,\di}$ is an m-$\lang$-lattice. Moreover, if $\alg{A}$ belongs to a class $\K$ of $\lang$-lattices closed under taking subalgebras and direct powers, then $\tuple{\alg{B},\bo,\di}\in\mK$.
\end{proposition}
\begin{proof}
It is straightforward to check that $\tuple{B,\mt,\jn,\bo,\di}$ satisfies the equations ${\rm (L1_\bo)}$--${\rm (L3_\bo)}$ and  ${\rm (L1_\di)}$--${\rm (L3_\di)}$ and is hence an m-lattice. To show that $\tuple{\alg{B},\bo,\di}$ is an m-$\lang$-lattice --- and therefore, if $\alg{A}$ belongs to a class $\K$ of $\lang$-lattices closed under taking subalgebras and direct powers, a member of $\mK$ --- observe that for any $n\in\N$, $\ops\in\lang_n$, $f_1,\dots,f_n\in B$, and $u\in W$, 
\begin{align*}
\bo(\ops(\bo f_1,\dots,\bo f_n))(u)&=\bigwedge_{v\in W}\ops(\bo f_1,\dots,\bo f_n)(v)\\
&=\bigwedge_{v\in W}\ops(\bo f_1(v),\dots,\bo f_n(v))\\
&=\ops(\bo f_1(u),\dots,\bo f_n(u))\\
&=\ops(\bo f_1,\dots,\bo f_n)(u),
\end{align*}
noting that in the third equality we have used the fact that $\bo f_i(v)=\bo f_i(u)$ for all $v\in W$ and $i\in\{1,\dots,n\}$.
\end{proof}

Let us call an m-$\lang$-lattice $\tuple{\alg{B},\bo,\di}$ {\em $\tuple{\alg{A},W}$-functional} if it is constructed as described in Proposition \ref{p:standard} for some $\lang$-lattice $\alg{A}$ and set $W$. Consider any class of $\lang$-lattices  $\K$. We call an m-$\lang$-lattice {\em $\K$-functional} if it is isomorphic to an $\tuple{\alg{A},W}$-functional m-$\lang$-lattice for some $\alg{A}\in\K$ and set $W$, omitting the prefix $\K$- if the class is clear from the context. 

The following result identifies the semantics of one-variable first-order logics with evaluations into functional m-$\lang$-lattices.

\begin{proposition}
Let $\alg{A}$ be any $\lang$-lattice.
\begin{enumerate}[itemsep=1mm]
\item[(a)] Let  $\model{S}=\tuple{S,\mfram{I}}$ be any $\alg{A}$-structure. Then $B\coloneqq\{\semvalue{\f}^\model{S}\mid\f\in\ofml(\lang)\}$ forms an $\tuple{\alg{A},S}$-functional m-$\lang$-lattice $\alg{B}$ and the $\alg{B}$-evaluation $g^\model{S}$, defined by $g^\model{S}(p_i)\coloneqq\mfram{I}(P_i)$ for each $i\in\N$, satisfies for all $\f,\psi\in\ofml(\lang)$, 
\begin{align*}
g^\model{S}(\f^\ast)=\semvalue{\f}^\model{S}
\quad\text{and}\quad
\model{S}\models\f\eq\psi\iff g^\model{S}(\f^\ast)=g^\model{S}(\psi^\ast).
\end{align*}
\item[(b)] Let $\alg{B}$ be any $\tuple{\alg{A},W}$-functional m-$\lang$-lattice for some set $W$, and let $e$ be any $\alg{B}$-evaluation. Then $\model{W}=\tuple{W,\mfram{J}}$, where  $\mfram{J}(P_i)\coloneqq e(p_i)$ for each $i\in\N$, is an $\alg{A}$-structure satisfying for all $\f,\psi\in\ofml(\lang)$,
\begin{align*}
e(\f^\ast)=\semvalue{\f}^\model{W}\quad\text{and}\quad\model{W}\models\f\eq\psi\iff e(\f^\ast)=e(\psi^\ast).
\end{align*}
\end{enumerate}
\end{proposition}
\begin{proof}
(a) To show that $\alg{B}$ is $\tuple{\alg{A},S}$-functional, it suffices to observe that for any $\semvalue{\f}^\model{S}\in B$, since $\model{S}$ is an $\alg{A}$-structure, the elements $\bigwedge\{\semvalue{\f}^\model{S}(v)\mid v\in S\}$ and $\bigvee\{\semvalue{\f}^\model{S}(v)\mid v\in S\}$ exist in $\alg{A}$ and correspond to the constant functions $\semvalue{\All{x}\f}^\model{S}\in B$ and  $\semvalue{\Exi{x}\f}^\model{S}\in B$, respectively. The fact that $g^\model{S}(\f^\ast)=\semvalue{\f}^\model{S}$ for all $\f\in\ofml(\lang)$, follows by a straightforward induction on the definition of $\f$, from which it follows directly also that $\model{S}\models\f\eq\psi\iff g^\model{S}(\f^\ast)=g^\model{S}(\psi^\ast)$, for all $\f,\psi\in\ofml(\lang)$.

(b) Since $\alg{B}$ is $\tuple{\alg{A},W}$-functional, the elements $\bigwedge_{v\in W}f(v)$ and $\bigvee_{v\in W}f(v)$ exist in $\alg{A}$ for every $f\in B$. We prove that $e(\f^\ast)=\semvalue{\f}^\model{W}$, by induction on the definition of $\f$,  from which it follows immediately that  $\model{W}=\tuple{W,\mfram{J}}$ is an $\alg{A}$-structure and $\model{W}\models\f\eq\psi\iff e(\f^\ast)=e(\psi^\ast)$, for all $\f,\p\in\ofml(\lang)$. In particular, for the case where $\f=\All{x}\p$, using the induction hypothesis for the second line,
\begin{align*}
\semvalue{\All{x}\p}^\model{W}(u)
&=\bigwedge\{\semvalue{\p}^\model{W}(v)\mid v\in W\}\\
&=\bigwedge\{e(\p^\ast)(v)\mid v\in W\}\\
&=\bo e(\p^\ast)(u)\\
&=e((\All{x}\p)^\ast)(u).
\end{align*}
The case where $\f=\Exi{x}\p$ is very similar.
\end{proof}

As a direct consequence of this theorem, we obtain the following relationship between consequence in the first-order logic defined over a (suitable) class $\K$ of $\lang$-lattices and consequence in the variety $\mK$.

\begin{corollary}\label{c:soundness2}
For any class $\K$ of $\lang$-lattices closed under taking subalgebras and direct powers, and set of $\ofml(\lang)$-equations $T\cup\{\f\eq\psi\}$,
\begin{align*}
T^\ast\mdl{\mK}\f^\ast\eq\psi^\ast\quad\Longrightarrow\quad T\fosc{\K}\f\eq\psi.
\end{align*}
Moreover, if every member of $\mK$ is $\K$-functional
\begin{align*}
T^\ast\mdl{\mK}\f^\ast\eq\psi^\ast\quad\Longleftrightarrow\quad T\fosc{\K}\f\eq\psi.
\end{align*}
\end{corollary}

Let us remark that a stricter notion of a functional algebra for a class $\K$ of $\lang$-lattices is considered in~\cite{BH02,CMT22} that coincides in our setting with the notion of being {\em $\K^c$-functional}, where  $\K^c$ is the class of complete members of $\K$. That is, an m-$\lang$-lattice $\tuple{\alg{B},\bo,\di}$ is $\K^c$-functional if it is isomorphic to a subalgebra of $\tuple{\alg{A}^W,\bo,\di}$ for some complete $\lang$-lattice $\alg{A}\in\K$ and set $W$, where $\bo$ and $\di$ are defined as described in Proposition \ref{p:standard}. 


\section{A functional representation theorem}\label{s:functional}

Adapting the proof of a similar result for Heyting Algebras~\cite[Theorem~3.6]{BH02}, we prove in this section that if a variety $\V$ of $\lang$-lattices has the superamalgamation property, then every member of $\mV$ is $\V$-functional, and hence, by Corollary~\ref{c:soundness2}, consequence in the one-variable first-order logic defined over $\V$ corresponds to consequence in $\mV$. 

We first recall the necessary algebraic notions. Let $\K$ be a class of $\lang$-lattices. A  {\em V-formation} in $\K$ is a $5$-tuple $\tuple{\alg{A},\alg{B}_1,\alg{B}_2,f_1,f_2}$  consisting of $\alg{A},\alg{B}_1,\alg{B}_2\in \K$ and embeddings $f_1\colon\alg{A}\to\alg{B}_1$,  $f_2\colon\alg{A}\to\alg{B}_2$. An {\em amalgam} in $\K$ of a V-formation $\tuple{\alg{A},\alg{B}_1,\alg{B}_2,f_1,f_2}$  in $\K$ is a triple $\tuple{\alg{C},g_1,g_2}$ consisting of $\alg{C}\in\K$ and embeddings $g_1\colon\alg{B}_1\to\alg{C}$, $g_2\colon\alg{B}_2\to\alg{C}$ such that $g_1\circ f_1 = g_2\circ f_2$; it is called a {\em superamalgam} if also for any $b_1\in B_1$, $b_2\in B_2$ and distinct $i,j\in\{1,2\}$,
\begin{align*}
g_i(b_i)\le g_j(b_j) \enspace\Longrightarrow\enspace g_i(b_i)\le g_i\circ f_i(a) = g_j\circ f_j(a)  \le g_j(b_j)\,\text{ for some }a\in A.
\end{align*}
The class $\K$ is said to have the {\em superamalgamation property} if every V-formation in $\K$ has a superamalgam in~$\K$.

\begin{theorem}\label{t:functional}
Let $\K$ be a class of $\lang$-lattices that is closed under taking direct limits and subalgebras, and has the superamalgamation property. Then every member of $\K$ is functional. 
\end{theorem}

\begin{proof}
Consider any $\tuple{\alg{A},\bo,\di}\in\mK$. Then $\alg{A}\in\K$ and, since $\K$ is closed under taking subalgebras, also $\bo\alg{A}\in\K$. We let $W:=\N^{>0}$ and define inductively a sequence of $\lang$-lattices $\tuple{\alg{A}_i}_{i\in W}$ in $\K$ and sequences of $\lang$-lattice embeddings $\tuple{f_i\colon\bo\alg{A}\to\alg{A}_i}_{i\in W}$, $\tuple{g_i\colon\alg{A}\to\alg{A}_i}_{i\in W}$, $\tuple{s_i\colon \alg{A}_{i-1}\to\alg{A}_i}_{i\in W}$. 

Let $\alg{A}_0\coloneqq\alg{A}$ and let $f_0\colon\bo\alg{A}\to\alg{A}$ be the inclusion map. For each $i\in W$, there exists inductively, by assumption, a superamalgam $\tuple{\alg{A}_i,s_i,g_i}$ of the V-formation $\tuple{\bo\alg{A},\alg{A}_{i-1},\alg{A},f_{i-1},f_0}$, and we define also $f_i\coloneqq s_i\circ f_{i-1}=g_i\circ f_0=g_i|_{\bo A}$. 

Now let $\alg{L}$ be the direct limit of the system $\tuple{\tuple{\alg{A}_i,s_i}}_{i\in W}$ with an associated sequence of $\lang$-lattice embeddings $\tuple{l_i\colon \alg{A}_i\to \alg{L}}_{i\in W}$. Since $\K$ is closed under taking direct limits, $\alg{L}$ belongs to $\K$. The first two superamalgamation steps of this construction are depicted in the following diagram:
\begin{center}
\begin{tikzpicture}
\coordinate[label=below:$\Box\mathbf{A}$](Box A1) at (0,0.5);
\coordinate[label=below:$\mathbf{A}$](Box A2) at (2.5,0.5);
\coordinate[label=below:$\mathbf{A}$](A) at (0,-1.6);
\coordinate[label=below:$\mathbf{A}_1$](A1) at (2.6,-1.6);
\coordinate[label=below:$\mathbf{A}_2$](A2) at (5.2,-1.6);
\coordinate[label=below:$\mathbf{A}_3$](A3) at (7.7,-1.6);
\coordinate[label=below:$\mathbf{L}$] (L) at (5.2,-3.7);
\coordinate[label=below:$\cdots$] (dots) at (8.5,-1.7);
\draw[->] (0,-0.2) -- (0,-1.47) node[midway, right]{$f_0$};
\draw[->] (0.5,0.2) -- (2.1,0.2) node[midway, above]{$f_0$};
\draw[->] (0.5,-0.2) -- (2.1,-1.47) node[midway, above right]{$f_1$};
\draw[->] (3,0.2) -- (7.2,-1.47) node[midway, above right]{$g_3$};
\draw[->] (3,-0.2) -- (4.7,-1.47) node[midway, above right]{$g_2$};
\draw[->] (2.5,-0.2) -- (2.5,-1.47) node[midway, right]{$g_1$};
\draw[->] (0.5,-1.88) -- (2.1,-1.88) node[midway, above]{$s_1$};
\draw[->] (3,-1.88) -- (4.7,-1.88) node[midway, above]{$s_2$};
\draw[->] (5.6,-1.88) -- (7.2,-1.88) node[midway, above]{$s_3$};
\draw[->] (3,-2.25) -- (4.85,-3.65) node[midway, above right]{$l_1$};
\draw[->] (5.2,-2.25) -- (5.2,-3.65) node[midway, right]{$l_2$};
\draw[->] (7.7,-2.25) -- (5.6,-3.65) node[midway, above left]{$l_3$};
\end{tikzpicture}
\end{center}
Since the operations of $\alg{L}^W$ are defined pointwise, $B\coloneqq\{\tuple{l_i\circ g_i(a)}_{i\in W}\mid a\in A\}$ is the universe of a subalgebra $\alg{B}$ of $\alg{L}^W$. We can also show that for each $a\in A$, the elements
\begin{align*}
\bigwedge_{j\in W}l_j\circ g_j(a)\qquad\text{and}\qquad\bigvee_{j\in W}l_j\circ g_j(a)
\end{align*}
exist in $L$ and hence that $\tuple{\alg{B},\bo,\di}$, with $\bo$ and $\di$ defined in Proposition \ref{p:standard}, is an $\tuple{\alg{L},W}$-functional m-$\lang$-lattice. Let $a\in A$ and fix some $i\in W$. It suffices to  show that $l_i\circ g_i(\bo a)$ and $l_i\circ g_i(\di a)$ are the greatest lower bound and  least upper bound, respectively, of $S\coloneqq\{l_j\circ g_j(a)\mid j\in W\}$. Observe first that for any $k\in W$,
\begin{align*}
l_k\circ g_k(\bo a)=l_k\circ f_k(\bo a)=l_{k+1}\circ s_{k+1}\circ f_k(\bo a)=l_{k+1}\circ g_{k+1}(\bo a),
\end{align*}
where the first and last equations follow from the definition of $f_k$ and the second follows from the fact that $\alg{L}$ is a direct limit. Hence for each $j\in W$,
\begin{align*}
l_i\circ g_i(\bo a)=l_j\circ g_j(\bo a)\leq l_j\circ g_j(a).
\end{align*}
So $l_i\circ g_i(\bo a)$ is a lower bound of $S$. Now suppose that $c\in L$ is another lower bound of $S$. Since $\alg{L}$ is a direct limit, there exist $k\in W$ and $d\in A_k$ such that
\begin{align*}
l_{k+1}\circ s_{k+1}(d)=l_k(d)=c\leq l_{k+1}\circ g_{k+1}(a).
\end{align*}

Since $l_{k+1}$ is an embedding, $s_{k+1}(d)\leq g_{k+1}(a)$. Hence, since $\tuple{\alg{A}_{k+1},s_{k+1},g_{k+1}}$ is a superamalgam of $\tuple{\bo\alg{A},\alg{A}_k,\alg{A},f_k,f_0}$, there exists $b\in\bo A$ such that
\begin{align*}
s_{k+1}(d)\leq s_{k+1}\circ f_k(b)=g_{k+1}\circ f_0(b)\leq g_{k+1}(a).
\end{align*}
But $s_{k+1}$ and $g_{k+1}$ are embeddings and $f_0$ is the inclusion map, so $d\leq f_k(b)$ and $b\leq a$. The latter inequality together with $b\in\bo A$, yields $b=\bo b\leq\bo a$. Hence also $f_k(b)\leq f_k(\bo a)=g_k(\bo a)$, and, using the first inequality,
\begin{align*}
c=l_k(d)\leq l_k\circ f_k(b)\leq l_k\circ g_k(\bo a)=l_i\circ g_i(\bo a).
\end{align*}
So $\bigwedge_{j\in W} l_j\circ g_j(a)=l_i\circ g_i(\bo a)$ exists in $L$ and the constant function $\tuple{l_i\circ g_i(\bo a)}_{i\in W}$ belongs to $B$. Also, symmetrically, $\bigvee_{j\in W}l_j\circ g_j(a)=l_i\circ g_i(\di a)$ exists in $L$ and the constant function $\tuple{l_i\circ g_i(\di a)}_{i\in W}$ belongs to $B$.

To show that $\tuple{\alg{A},\bo,\di}$ is functional, it remains to prove that the following map is an isomorphism:
\begin{align*}
f\colon\tuple{\alg{A},\bo,\di}\to\tuple{\alg{B},\bo,\di};\quad a\mapsto \tuple{l_i\circ g_i(a)}_{i\in W}.
\end{align*}
Since the operations of $\alg{L}^W$ are defined pointwise, it is easily checked that $f$ is an $\lang$-lattice isomorphism. Moreover, recalling that $l_i\circ g_i(\bo a)=\bigwedge_{j\in W} l_i\circ g_i(a)$ for each $a\in A$, it follows that
\begin{align*}
f(\bo a)=\tuple{l_i\circ g_i(\bo a)}_{i\in W}
=\tuple{\bigwedge_{j\in W} l_j\circ g_j(a)}_{i\in W}
=\bo\tuple{l_i\circ g_i(a)}_{i\in W}
=\bo f(a),
\end{align*}
and, similarly, $f(\di a)=\di f(a)$ for all $a\in A$. 
\end{proof}

Combining Theorem~\ref{t:functional} with Corollary~\ref{c:soundness2} yields the following result.

\begin{corollary}\label{c:completeness}
If $\V$ is a variety of $\lang$-lattices that has the superamalgamation property, then for any set $T\cup\{\f\eq\p\}$ of $\ofml(\lang)$-equations,
\begin{align*}
T\fosc{\V}\f\eq\p \quad\Longleftrightarrow\quad T^\ast\mdl{\mV}\f^\ast \eq\p^\ast.
\end{align*}
\end{corollary}

\begin{example}\label{e:lattices}
The variety of lattices has the superamalgamation property~\cite{Gra98}. Hence, by Theorem~\ref{t:functional}, every m-lattice is functional, and consequence in the one-variable first-order lattice logic corresponds to consequence in m-lattices.
\end{example}

\begin{example}
$\cls{FL_e}$, $\cls{FL_{ew}}$, and $\cls{FL_{ec}}$, and many other varieties of $\cls{FL_e}$-algebras have the superamalgamation property, which is equivalent in this setting to the Craig interpolation property for the associated substructural logic (see,~e.g.,~\cite{GJKO07}). Hence, for any such variety $\V$ --- notably, for $\V\in\{\cls{FL_e},\cls{FL_{ew}},\cls{FL_{ec}}\}$ --- every member of $\mV$ is functional, and consequence in the one-variable first-order substructural logic defined over $\V$ corresponds to consequence in $\mV$.
\end{example}

\begin{example}\label{e:modal}
A normal modal logic has the Craig interpolation property if and only if the associated variety of modal algebras --- Boolean algebras with an operator --- has the superamalgamation property~\cite{Mak92}. Moreover, there exist infinitely many such logics~\cite{Rau82}, including well-known cases such as $\lgc{K}$, $\lgc{KT}$, $\lgc{K4}$, and $\lgc{S4}$. Hence our results yield axiomatizations for the one-variable fragments of infinitely many first-order logics defined over varieties of modal algebras.
\end{example}

Suppose finally that $\K$ is a class of $\lang$-lattices that is not only closed under taking direct limits and subalgebras and has the superamalgamation property, but also admits regular completions. In this case, we can adapt the proof of Theorem~\ref{t:functional} to show that every member of $\K$ is $\K^c$-functional, which --- as noted at the end of Section~\ref{s:algebraic} --- corresponds to the stricter notion of a functional algebra considered in~\cite{BH02,CMT22}. Just observe that, given some $\tuple{\alg{A},\bo,\di}\in\mK$, the direct limit $\alg{L}\in\K$ constructed  in the proof embeds into some $\alg{\bar{L}}\in\K^c$ and hence, reasoning as before, $\tuple{\alg{A},\bo,\di}$ is isomorphic to a subalgebra of $\tuple{\alg{\bar{L}}^W,\bo,\di}$.


\section{A proof-theoretic strategy}\label{s:prooftheory}

In this section, we describe an alternative proof-theoretic strategy for establishing completeness of axiomatizations for one-variable fragments of first-order logics. The key step is to prove that a derivation of a one-variable formula in a sequent calculus for the first-order logic can be transformed into a derivation that uses just one variable. To illustrate, we consider the first-order version of the full Lambek calculus with exchange $\lgc{FL_e}$,  then extend the method to a broader family of first-order substructural logics.

The one-variable fragment of the first-order version of $\lgc{FL_e}$ can be presented as a cut-free sequent calculus. This presentation has the advantage that although a derivation of a one-variable formula in the calculus may use more than one variable, it will not introduce any new occurrences of quantifiers. We therefore consider the set $\ofmls(\langs)$ of first-order formulas built inductively from unary predicates $\{P_i\}_{i\in\N}$, variables $\{x\}\cup\{x_i\}_{i\in\N}$, connectives in $\lang_s$, and quantifiers $\All{x}$ and $\Exi{x}$, such that no occurrence of a variable $x_i$  lies in the scope of a quantifier. Clearly, $\ofml(\langs)\subseteq\ofmls(\langs)$. We write $\f(\bar{w})$ to denote that the free variables of $\f\in\ofmls(\langs)$ belong to the set $\bar{w}$, and indicate by writing $\f(\bar{w},y)$ that $y\not\in\bar{w}$.

For the purposes of this paper, we define a {\em sequent} to be an ordered pair of finite multisets of formulas in $\ofmls(\langs)$, denoted by $\Ga\seq\De$, such that $\De$ contains at most one $\langs$-formula.\footnote{The full Lambek calculus with exchange is typically presented using sequents consisting of finite {\em sequences} of formulas and an ``exchange rule'' for permuting formulas (see, e.g.,~\cite{GJKO07,MPT23}).} As usual, we denote the multiset sum of two finite multisets of formulas $\Ga_1$ and $\Ga_2$ by $\Ga_1,\Ga_2$, and the empty multiset by an empty space. We also define, for $n\in\N^{>0}$ and $\f_1,\dots,\f_n,\p\in\ofmls(\langs)$,
\begin{align*}
\textstyle
\prod(\f_1,\dots,\f_n)\coloneqq\f_1\cdots\f_n, \quad
\prod()\coloneqq\ut, \quad
\sum(\p)\coloneqq\p, \quad
\sum()\coloneqq\zr.
\end{align*}
The sequent calculus $\fFLe$ is displayed in Figure~\ref{f:fFLe}, where the quantifier rules are subject to the following side-conditions:
\begin{enumerate}[itemsep=1mm]
\item[\rm (i)] if the conclusion of an application of  $\falr$ or $\err$ contains at least one free occurrence of a variable, then the variable $u$ occurring in the premise also occurs freely in the conclusion;
\item[\rm (ii)]  the variable $y$ occurring in the premise of $\farr$ and $\elr$ does not occur freely in the conclusion of the rule.
\end{enumerate}
If there exists a derivation $d$ of a sequent $\Ga\seq\De$ in a sequent calculus $\lgc{C}$, we write $d\der{\lgc{C}}\Ga\seq\De$ or simply $\der{\lgc{C}}\Ga\seq\De$. 

\begin{figure}[t]
\centering
\fbox{
 \parbox{.95\linewidth}{
 \[
\begin{array}{c}
\text{Axioms}\\[.2in]
\begin{array}{ccccc}
\infer[\pfa{\idr}]{\f \seq \f}{} & \qquad & \infer[\pfa{\flr}]{\zr \seq}{} & \qquad & \infer[\pfa{\trr}]{\seq\ut}{}\
\end{array}\\[.25in]
\text{Operation Rules}\\[.2in]
\begin{array}{ccc}
\infer[\pfa{\tlr}]{\Ga,  \ut \seq \De}{\Ga \seq \De} & \quad &\infer[\pfa{\frr}]{\Ga \seq\zr}{\Ga \seq}\\[.15in]
\infer[\pfa{\ilr}]{\Ga_1, \Ga_2, \f \to \p \seq \De}{\Ga_1 \seq \f & \Ga_2, \p \seq \De} & \quad & 
\infer[\pfa{\irr}]{\Ga \seq \f \to \p}{\Ga, \f \seq \p}\\[.15in]
\infer[\pfa{\pdlr}]{\Ga, \f \pd \p \seq \De}{\Ga,\f,\p \seq \De} & \quad &
\infer[\pfa{\pdrr}]{\Ga_1, \Ga_2 \seq \f \pd \p}{\Ga_1 \seq \f & \Ga_2 \seq \p}\\[.15in]
\infer[\pfa{\alr_1}]{\Ga, \f \mt \p \seq \De}{\Ga, \f \seq \De} & & 
\infer[\pfa{\orr_1}]{\Ga \seq \f \jn \p}{\Ga \seq \f}\\[.15in]
\infer[\pfa{\alr_2}]{\Ga, \f \mt \p \seq \De}{\Ga, \p \seq \De} & & 
\infer[\pfa{\orr_2}]{\Ga \seq \f \jn \p}{\Ga \seq \p}\\[.15in]
\infer[\pfa{\olr}]{\Ga, \f \jn \p \seq \De}{\Ga, \f \seq \De & \Ga, \p \seq \De} & & 
\infer[\pfa{\arr}]{\Ga \seq \f \mt \p}{\Ga \seq \f & \Ga \seq \p}\\[.15in]
\infer[\pfa{\falr}]{\Ga, \All{x}\f(x) \seq\De}{\Ga,\f(u)\seq\De} & & 
\infer[\pfa{\farr}]{\Ga\seq \All{x}\p(x)}{\Ga\seq\p(y)}\\[.15in]
\infer[\pfa{\elr}]{\Ga,\Exi{x}\f(x) \seq\De}{\Ga,\f(y)\seq\De} & & 
\infer[\pfa{\err}]{\Ga\seq \Exi{x}\p(x)}{\Ga\seq\p(u)}
\end{array}
\end{array}
\]}}
\caption{The Sequent Calculus $\fFLe$}
\label{f:fFLe}
\end{figure}

The following relationship between derivability of sequents in $\fFLe$ and (first-order) validity of equations in the variety $\cls{FL_e}$ is a direct consequence of the completeness of a cut-free sequent calculus for the first-order version of $\lgc{FL_e}$.

\begin{proposition}[cf.~\cite{OK85,Kom86}]\label{p:ono} 
For any sequent $\Ga\seq\De$ containing formulas from $\ofml$,
\begin{align*}
\textstyle\der{\fFLe}\Ga\seq\De \quad\Longleftrightarrow\quad\: \fosc{\cls{FL_e}}\prod\Ga\le\sum\De.
\end{align*}
\end{proposition}

We now establish an interpolation property for the calculus $\fFLe$. For any derivation $d$ of a sequent in $\fFLe$, let $\md(d)$ denote the maximum number of applications of the rules $\farr$ and $\elr$ that occur on a branch of $d$.

\begin{lemma}\label{l:interpolation}
If $d\der{\fFLe}\Ga(\bar{w},y),\Pi(\bar{w},z)\seq\De(\bar{w},z)$, with $y\neq z$ and $x\not\in\bar{w}\cup\{y,z\}$, then there exist $\x(\bar{w})\in\ofmls(\langs)$ and derivations $d_1,d_2$ in $\fFLe$ such that $\md(d_1),\md(d_2)\le\md(d)$ and
\begin{align*}
d_1\der{\fFLe}\Ga(\bar{w},y)\seq\x(\bar{w}), \quad d_2\der{\fFLe}\Pi(\bar{w},z),\x(\bar{w})\seq\De(\bar{w},z).
 \end{align*}
\end{lemma}
\begin{proof}
We prove the claim by induction on the height of the derivation $d$ in $\fFLe$ of $\Ga(\bar{w},y),\Pi(\bar{w},z)\seq\De(\bar{w},z)$, considering in turn the last rule applied in the derivation. Note first that if $y$ does not occur in $\Ga$, we can define $\x(\bar{w}):=\prod\Ga$, and obtain a derivation $d_1$ of $\Ga(\bar{w},y)\seq\x(\bar{w})$, ending with repeated applications of $\pdrr$ and $\trr$, and a derivation $d_2$ of $\Pi(\bar{w},z),\x(\bar{w})\seq\De(\bar{w},z)$ extending $d$ with repeated applications of $\pdlr$ and $\tlr$, such that $\md(d_1)=0$ and $\md(d_2)\le\md(d)$. Similarly, if $z$ does not occur in $\Pi,\De$, we can define $\x(\bar{w}):=\prod\Pi\to\sum\De$, and obtain a derivation $d_1$ of $\Ga(\bar{w},y)\seq\x(\bar{w})$ that extends $d$ with repeated applications of $\pdlr$, $\tlr$, and $\frr$, followed by an application of $\irr$, and a derivation $d_2$ of $\Pi(\bar{w},z),\x(\bar{w})\seq\De(\bar{w},z)$ ending with repeated applications of $\pdrr$, $\trr$, and $\flr$, followed by an application of $\ilr$, such that $\md(d_1)=\md(d)$ and $\md(d_2)=0$.

For the base cases where $d$ ends with $\idr$, $\trr$, or $\flr$, either  $y$ does not occur in $\Ga$ or $z$ does not occur in $\Pi,\De$. For the remainder of the proof, let us assume without further comment that $y$ occurs in $\Ga$ and $z$ occurs in $\Pi,\De$. The cases where $d$ ends with an operational rule for one of the propositional connectives are all straightforward, so let us just consider $\ilr$ as an example.

Suppose first that $\Ga(\bar{w},y)$ is $\Ga_1(\bar{w},y),\Ga_2(\bar{w},y),\f(\bar{w},y)\to\p(\bar{w},y)$ and $\Pi(\bar{w},z)$ is $\Pi_1(\bar{w},z),\Pi_2(\bar{w},z)$, and 
\begin{align*}
d'_1&\der{\fFLe}\Ga_1(\bar{w},y),\Pi_1(\bar{w},z)\seq\f(\bar{w},y),\\
d'_2&\der{\fFLe}\Ga_2(\bar{w},y),\p(\bar{w},y),\Pi_2(\bar{w},z)\seq\De(\bar{w},z).
\end{align*}
Two applications of the induction hypothesis yield formulas $\x_1(\bar{w}),\x_2(\bar{w})$ and derivations $d'_{11},d'_{12},d'_{21},d'_{22}$ such that 
\begin{align*}
d'_{11}\der{\fFLe}\Ga_1(\bar{w},y),\x_1(\bar{w})\seq\f(\bar{w},y), &  \quad d'_{12}\der{\fFLe}\Pi_1(\bar{w},z)\seq\x_1(\bar{w}),\\
d'_{21}\der{\fFLe}\Ga_2(\bar{w},y),\p(\bar{w},y)\seq\x_2(\bar{w}), & \quad d'_{22}\der{\fFLe}\Pi_2(\bar{w},z),\x_2(\bar{w})\seq\De(\bar{w},z).
\end{align*} 
Let $\x(\bar{w}):=\x_1(\bar{w})\to\x_2(\bar{w})$. Then $d'_{11},d'_{21}$, together with applications of $\ilr$ and $\irr$, and $d'_{12},d'_{22}$, together with an application of $\ilr$, yield derivations $d_1$ and $d_2$, respectively,  such that
\begin{align*}
& d_{1}\der{\fFLe}\Ga_1(\bar{w},y),\Ga_2(\bar{w},y),\f(\bar{w},y)\to\p(\bar{w},y)\seq\x_1(\bar{w})\to\x_2(\bar{w}),\\
& d_{2}\der{\fFLe}\Pi_1(\bar{w},z),\Pi_2(\bar{w},z),\x_1(\bar{w})\to\x_2(\bar{w})\seq\De(\bar{w},z).
\end{align*}
Clearly, the constraints on $\md (d_1)$ and $\md(d_2)$ are satisfied.

Now suppose that $\Ga(\bar{w},y)$ and $\Pi(\bar{w},z)$ are of the form $\Ga_1(\bar{w},y),\Ga_2(\bar{w},y)$ and $\Pi_1(\bar{w},z),\Pi_2(\bar{w},z),\f(\bar{w},z)\to\p(\bar{w},z)$, respectively, and 
\begin{align*}
d'_1&\der{\fFLe}\Ga_1(\bar{w},y),\Pi_1(\bar{w},z)\seq\f(\bar{w},z),\\
d'_2&\der{\fFLe}\Ga_2(\bar{w},y),\Pi_2(\bar{w},z),\p(\bar{w},z)\seq\De(\bar{w},z).
\end{align*}
Two applications of the induction hypothesis yield formulas $\x_1(\bar{w}),\x_2(\bar{w})$ and derivations $d'_{11},d'_{12},d'_{21},d'_{22}$ such that 
\begin{align*}
d'_{11}\der{\fFLe}\Ga_1(\bar{w},y)\seq\x_1(\bar{w}), & \quad d'_{12}\der{\fFLe}\Pi_1(\bar{w},z),\x_1(\bar{w})\seq\f(\bar{w},z),\\
d'_{21}\der{\fFLe}\Ga_2(\bar{w},y)\seq\x_2(\bar{w}), & \quad d'_{22}\der{\fFLe}\Pi_2(\bar{w},z),\p(\bar{w},z),\x_2(\bar{w})\seq\De(\bar{w},z).
\end{align*} 
Let $\x(\bar{w}):=\x_1(\bar{w})\pd\x_2(\bar{w})$. Then $d'_{11},d'_{21}$, together with an application of $\pdrr$, and $d'_{12},d'_{22}$, together with applications of $\ilr$ and $\pdlr$, yield derivations $d_1$ and $d_2$, respectively,  such that
\begin{align*}
 & d_{1}\der{\fFLe}\Ga_1(\bar{w},y),\Ga_2(\bar{w},y)\seq\x_1(\bar{w})\pd\x_2(\bar{w}),\\
& d_{2}\der{\fFLe}\Pi_1(\bar{w},z),\Pi_2(\bar{w},z),\f(\bar{w},z)\to\p(\bar{w},z),\x_1(\bar{w})\pd\x_2(\bar{w})\seq\De(\bar{w},z).
 \end{align*}
Again, the constraints on $\md(d_1)$ and $\md(d_2)$ are clearly satisfied.

Next, we consider all cases where $d$ ends with an application of one of the quantifier rules.

\begin{enumerate}[font=\upshape, label={$\bullet$},  itemsep=1mm]

\item
$\falr$: Suppose first that $\Ga(\bar{w},y)$ is $\Ga'(\bar{w},y),\All{x}\f(x)$ and
\begin{align*}
d'\der{\fFLe}\Ga'(\bar{w},y),\f(u),\Pi(\bar{w},z)\seq\De(\bar{w},z),
\end{align*}
where $\md(d')=\md(d)$. For subcase (i), suppose that $u\in\bar{w}\cup\{y\}$. By the induction hypothesis, there exist a formula $\x(\bar{w})$ and derivations $d_1',d_2$ such that $\md(d_1'),\md(d_2)\leq\md(d')$ and
\begin{align*}
d_1'\der{\fFLe}\Ga'(\bar{w},y),\f(u)\seq\x(\bar{w}),\quad d_2\der{\fFLe}\Pi(\bar{w},z),\x(\bar{w})\seq\De(\bar{w},z).
\end{align*} 
Extending $d_1'$ with an application of $\falr$ yields a derivation $d_1$ such that $\md(d_1)=\md(d_1')\leq\md(d')=\md(d)$ and
\begin{align*}
d_1\der{\fFLe}\Ga'(\bar{w},y),\All{x}\f(x)\seq\x(\bar{w}).
\end{align*} 
For subcase (ii), suppose that $u=z$. By the induction hypothesis, there exists a formula $\x'(\bar{w})$ and derivations $d_1',d_2'$ such that $\md(d_1'),\md(d_2')\leq\md(d')$ and
\begin{align*}
d_1'\der{\fFLe}\Ga'(\bar{w},y)\seq\x'(\bar{w}), \quad
d_2'\der{\fFLe}\f(z),\Pi(\bar{w},z),\x'(\bar{w})\seq\De(\bar{w},z).
\end{align*} 
Let $\x(\bar{w})\coloneqq\x'(\bar{w})\cdot\All{x}\f(x)$. Combining  an instance $\All{x}\f(x)\seq\All{x}\f(x)$ of $\idr$ with $d_1'$ and an application of $\pdrr$ to $d_1'$ yields a derivation $d_1$ such that $\md(d_1)=\md(d_1')\leq\md(d')=\md(d)$ and
\begin{align*}
d_1\der{\fFLe}\Ga'(\bar{w},y),\All{x}\f(x)\seq\x'(\bar{w})\cdot\All{x}\f(x).
\end{align*} 
Also, $d_2'$ extended with applications of $\falr$ and $\pdlr$ yields a derivation $d_2$ such that $\md(d_2)=\md(d_2')\leq\md(d')=\md(d)$ and
\begin{align*}
d_2\der{\fFLe}\Pi(\bar{w},z),\x'(\bar{w})\cdot\All{x}\f(x)\seq\De(\bar{w},z).
\end{align*}
Now suppose that $\Pi(\bar{w},z)$ is $\Pi'(\bar{w},z),\All{x}\f(x)$ and
\begin{align*}
d'\der{\fFLe}\Ga(\bar{w},y),\Pi'(\bar{w},z),\f(u)\seq\De(\bar{w},z),
\end{align*}
where $\md(d')=\md(d)$. The case of $u\in\bar{w}\cup\{z\}$ is similar to subcase (i) above, so suppose $u=y$.  By the induction hypothesis, there exist a formula $\x'(\bar{w})$ and derivations $d_1',d_2'$ such that $\md(d_1'),\md(d_2')\leq\md(d')$ and
\begin{align*}
d_1'\der{\fFLe}\Ga(\bar{w},y),\f(y)\seq\x'(\bar{w}),\quad d_2'\der{\fFLe}\Pi'(\bar{w},z),\x'(\bar{w})\seq\De(\bar{w},z).
\end{align*}
Let $\x(\bar{w})\coloneqq\All{x}\f(x)\to\x'(\bar{w})$. Extending $d_1'$ with applications of $\falr$ and $\irr$ yields a derivation $d_1$ such that $\md(d_1)=\md(d_1')\leq\md(d')=\md(d)$ and
\begin{align*}
d_1\der{\fFLe}\Ga(\bar{w},y)\seq\All{x}\f(x)\to\x'(\bar{w}).
\end{align*}
Also, $d_2'$ and $\All{x}\f(x)\seq\All{x}\f(x)$ combined with an application of $\ilr$ yields a derivation $d_2$ such that $\md(d_2)=\md(d_2')\leq\md(d')=\md(d)$ and
\begin{align*}
d_2\der{\fFLe}\Pi'(\bar{w},z),\All{x}\f(x),\All{x}\f(x)\to\x'(\bar{w})\seq\De(\bar{w},z).
\end{align*}
\item $\farr$: Suppose that $\De(\bar{w},z)$ is $\All{x}\f(x)$ and for some variable $u$ that does not occur freely in $\Ga(\bar{w},y),\Pi(\bar{w},z)\seq\All{x}\f(x)$,
\begin{align*}
d'\der{\fFLe}\Ga(\bar{w},y),\Pi(\bar{w},z)\seq\f(u),
\end{align*}
where $\md(d')=\md(d)-1$.
By the induction hypothesis, there exist a formula $\x'(\bar{w},u)$ and derivations $d'_{1},d'_{2}$ such that $\md(d'_1),\md(d'_2)\le\md(d')$ and
\begin{align*}
d_1'\der{\fFLe}\Ga(\bar{w},y)\seq\x'(\bar{w},u),\quad d'_2\der{\fFLe}\Pi(\bar{w},z),\x'(\bar{w},u)\seq\f(u).
\end{align*}
Let $\x(\bar{w})\coloneqq\All{x}\x'(\bar{w},x)$. Extending $d_1'$ with an application of $\farr$ yields a derivation $d_1$ such that $\md(d_1)=\md(d_1')+1\leq\md(d')+1=\md(d)$ and
\begin{align*}
d_1\der{\fFLe}\Ga(\bar{w},y)\seq\x(\bar{w}).
\end{align*}
Also, extending $d_2'$ with applications of $\falr$ and $\farr$ yield a derivation $d_2$ such that $\md(d_2)=\md(d'_2)+1\le\md(d')+1=\md(d)$ and
\begin{align*}
d_2\der{\fFLe}\Pi(\bar{w},z),\x(\bar{w})\seq\All{x}\f(x).
\end{align*}

\item $\err$: Suppose that $\De(\bar{w},z)$ is $\Exi{x}\f(x)$ and
\begin{align*}
d'\der{\fFLe}\Ga(\bar{w},y),\Pi(\bar{w},z)\seq\f(u),
\end{align*}
where $\md(d')=\md(d)$. For subcase (i), suppose that $u\in\bar{w}\cup\{z\}$. By the induction hypothesis, there exist a formula $\x(\bar{w})$ and derivations $d_1,d_2'$ such that $\md(d_1),\md(d_2')\leq\md(d')$ and
\begin{align*}
\der{\fFLe}\Ga(\bar{w},y)\seq\x(\bar{w}),\quad d_2'\der{\fFLe}\Pi(\bar{w},z),\x(\bar{w})\seq\f(u).
\end{align*}
Extending $d_2'$ with an application of $\err$ yields a derivation $d_2$ such that $\md(d_2)=\md(d_2')\leq\md(d')$ and
\begin{align*}
d_2\der{\fFLe}\Pi(\bar{w},z),\x(\bar{w})\seq\Exi{x}\f(x).
\end{align*}
For subcase (ii), suppose that $u=y$. By the induction hypothesis, there exists a formula $\x'(\bar{w})$ and derivations $d_1',d_2'$ such that $\md(d_1'),\md(d_2')\leq\md(d')$ and
\begin{align*}
d_1'\der{\fFLe}\Pi(\bar{w},z)\seq\x'(\bar{w}),\quad d_2'\der{\fFLe}\Ga(\bar{w},y),\x'(\bar{w})\seq\f(y).
\end{align*}
Let $\x(\bar{w})\coloneqq\x'(\bar{w})\to\Exi{x}\f(x)$. Combining $d_2'$ with applications of $\err$ and $\irr$ yields a derivation $d_1$ such that $\md(d_1)=\md(d_2')\leq\md(d')=\md(d)$ and
\begin{align*}
d_1\der{\fFLe}\Ga(\bar{w},y)\seq\x'(\bar{w})\to\Exi{x}\f(x).
\end{align*}
Also, combining the instance $\Exi{x}\f(x)\seq\Exi{x}\f(x)$ of $\idr$ and $d_1'$ with $\ilr$ yields a derivation $d_2$ such that $\md(d_2)=\md(d_1')\leq\md(d')=\md(d)$ and
\begin{align*}
d_2\der{\fFLe}\Pi(\bar{w},z),\x'(\bar{w})\to\Exi{x}\f(x)\seq\Exi{x}\f(x).
\end{align*}

\item	$\elr$: Suppose first that $\Ga(\bar{w},y)$ is $\Ga'(\bar{w},y),\Exi{x}\f(x)$ and for some variable $u$ that does not occur freely in $\Ga(\bar{w},y),\Pi(\bar{w},z)\seq\De(\bar{w},z)$,
\begin{align*}
d'\der{\fFLe}\Ga'(\bar{w},y),\f(u),\Pi(\bar{w},z)\seq\De(\bar{w},z),
\end{align*}
where $\md(d')=\md(d)-1$. By the induction hypothesis, there exist a formula $\x'(\bar{w},u)$ and derivations $d_1',d_2'$ such that $\md(d_1'),\md(d_2')\leq\md(d')$ and
\begin{align*}
\qquad d_1'\der{\fFLe}\Ga'(\bar{w},y),\f(u)\seq\x'(\bar{w},u),\quad d_2'\der{\fFLe}\Pi(\bar{w},z),\x'(\bar{w},u)\seq\De(\bar{w},z).
\end{align*}
We define $\x(\bar{w})$ to be $\Exi{u}\x'(\bar{w},u)$. Then applications of $\err$ and $\elr$ to $d_1$ yield a derivation $d_1$ such that $\md(d_1)=\md(d_1')+1\leq\md(d')+1=\md(d)$ and
\begin{align*}
d_1\der{\fFLe}\Ga'(\bar{w},y),\Exi{x}\f(x)\seq\x(\bar{w}).
\end{align*}
An application of $\elr$ yields a derivation $d_2$ satisfying $\md(d_2)=\md(d_2')+1\leq\md(d')+1=\md(d)$ and
\begin{align*}
d_2\der{\fFLe}\Pi(\bar{w},z),\x(\bar{w})\seq\De(\bar{w},z).
\end{align*}
Now suppose $\Pi(\bar{w},z)$ is $\Pi'(\bar{w},z),\Exi{x}\f(x)$ and for some variable $u$ that does not occur freely in $\Ga(\bar{w},y),\Pi(\bar{w},z)\seq\De(\bar{w},z)$,
\begin{align*}
d'\der{\fFLe}\Ga(\bar{w},y),\Pi'(\bar{w},z),\f(u)\seq\De(\bar{w},z),
\end{align*}
where $\md(d')=\md(d)-1$. By the induction hypothesis, there exist a formula $\x'(\bar{w},u)$ and derivations $d_1',d_2'$ such that $\md(d_1'),\md(d_2')\leq\md(d')$ and
\begin{align*}
\quad d_1'\der{\fFLe}\Ga(\bar{w},y)\seq\x'(\bar{w},u),\quad d_2'\der{\fFLe}\Pi'(\bar{w},z),\f(u),\x(\bar{w},u)\seq\De(\bar{w},z).
\end{align*}
Let $\x(\bar{w})\coloneqq\All{x}\x'(\bar{w},x)$. The derivation $d_1'$ together with an application of $\farr$ yields a derivation $d_1$ satisfying $\md(d_1)=\md(d_1')+1\leq\md(d')+1=\md(d)$ and
\begin{align*}
d_1\der{\fFLe}\Ga(\bar{w},y)\seq\x(\bar{w}).
\end{align*}
Then $d_2'$ together with applications of $\falr$ and $\elr$ yields a derivation $d_2$ satisfying $\md(d_2)=\md(d_2')+1\leq\md(d')+1=\md(d)$ and
$$d_2\der{\fFLe}\Pi(\bar{w},y),\Exi{x}\f(x),\x(\bar{w})\seq\De(\bar{w},z).\eqno\qed$$
\end{enumerate}
\noqed\end{proof}

Using this lemma we can now reprove using proof-theoretic means the special case of Corollary~\ref{c:completeness} for the variety $\cls{FL_e}$.

\begin{theorem}\label{t:FLecomplete}
For any set $T\cup\{\f\eq\p\}$ of $\ofml(\langs)$-equations,
\begin{align*}
T\fosc{\cls{FL_e}}\f\eq\p \quad\Longleftrightarrow\quad T^\ast\mdl{\cls{mFL_e}}\f^\ast \eq\p^\ast.
\end{align*}
\end{theorem}
\begin{proof}
The right-to-left direction follows directly from Corollary~\ref{c:soundness2}.  For the converse, note first that due to compactness and the local deduction theorem for $\fosc{\V}$ (see~\cite[Sections~4.6, 4.8]{CN21}), we can restrict to the case where $T = \emptyset$. Hence, by  Proposition~\ref{p:ono}, it suffices to prove that for any sequent $\Ga\seq\De$ consisting only of formulas from $\ofml(\langs)$, 
\begin{align*}
\textstyle
d\der{\fFLe}\Ga\seq\De
\quad\Longrightarrow\quad
\mdl{\cls{mFL_e}}(\prod\Ga)^\ast \le (\sum\De)^\ast.
\end{align*}
We proceed by induction on the lexicographically ordered pair $\tuple{\md(d),\height(d)}$, where $\height(d)$ is the height of the derivation $d$. The base cases are clear and the cases for the last application of a rule in $d$ except $\farr$ and $\elr$ all follow by applying the induction hypothesis and the equations defining $\cls{mFL_e}$. Just note that for each such application, the premises contain only formulas from $\ofml(\langs)$ with at least one fewer symbol. In particular, for $\falr$ and $\err$, it can be assumed that the variable $u$ occurring in the premise is $x$ and the result follows using {\rm (L1$_\bo$)} or {\rm (L1$_\di$)}.

Suppose now that the last rule applied in $d$ is $\farr$, where $\De$ is $\All{x}\p(x)$ and $x$ may occur freely in $\Ga$. Then $d'\der{\fFLe}\Ga\seq\p(z)$ with $\md(d')=\md(d)-1$, where $z$ is a variable distinct from $x$. We write $\Ga(y)$ and $d'(y)$ to denote $\Ga$ and $d'$ with all free occurrences of $x$ replaced by $y$. Clearly, $d'(y)\der{\fFLe}\Ga(y)\seq\p(z)$ with $\md(d'(y))=\md(d')$. Hence, by Lemma~\ref{l:interpolation}, there exist a sentence $\x$ and derivations $d_1,d_2$ such that $\md(d_1),\md(d_2)\le\md(d')$ and 
\begin{align*}
d_1\der{\fFLe}\Ga(y)\seq\x,\quad d_2\der{\fFLe}\x\seq\p(z).
\end{align*}
Since $\x$ is a sentence and $x$ does not occur freely in $\Ga(y)$ or $\p(z)$, we can assume that $d_1$ and $d_2$ do not contain any free occurrences of $x$, and, substituting all occurrences of $y$ in $d_1$, and $z$ in $d_2$, by $x$, obtain derivations $d'_1$ of $\Ga\seq\x$ and $d'_2$ of $\x\seq\p(x)$ with  $\md(d'_1)=\md(d_1)$ and $\md(d'_2)=\md(d_2)$. Hence, by the induction hypothesis twice, $\mdl{\cls{mFL_e}}(\prod\Ga)^\ast \le \x^\ast$ and $\mdl{\cls{mFL_e}} \x^\ast \le\p(x)^\ast$. Since $\x$ is a sentence,  the equations defining $\cls{mFL_e}$ yield also $\mdl{\cls{mFL_e}}\x^\ast \le (\All{x}\p(x))^\ast$. So  $\mdl{\cls{mFL_e}}(\prod\Ga)^\ast\le(\All{x}\p(x))^\ast$.

Suppose finally that the last rule applied in $d$ is $\elr$, where $\Ga$ is $\Pi,\Exi{x}\p(x)$ and $x$ may occur freely in $\Pi$ and $\De$. Then $d'\der{\fFLe}\Pi,\p(y)\seq\De$ with $\md(d')=\md(d)-1$, where $y$ is a variable distinct from $x$. We write $\Pi(z)$, $\De(z)$, and $d'(z)$ to denote $\Pi$, $\De$, and $d'$ with all free occurrences of $x$ replaced by $z$. Clearly, $d'(z)\der{\fFLe}\Pi(z),\p(y)\seq\De(z)$ with $\md(d'(z))=\md(d')$.  By Lemma~\ref{l:interpolation}, there exist a sentence $\x$ and derivations $d_1,d_2$ such that $\md(d_1),\md(d_2)\le\md(d')$ and
\begin{align*}
d_1\der{\fFLe}\p(y)\seq\x, \quad d_2\der{\fFLe}\Pi(z),\x\seq\De(z).
\end{align*}
Since $\x$ is a sentence and $x$ does not occur freely in $\p(y)$, $\Pi(z)$, or $\De(z)$, we can assume that $d_1$ and $d_2$ do not contain any free occurrences of $x$, and, substituting all occurrences of $y$ in $d_1$, and $z$ in $d_2$, by $x$, obtain derivations $d'_1$ of $\p(x)\seq\x$ and $d'_2$ of $\Pi,\x\seq\De$ with  $\md(d'_1)=\md(d_1)$ and $\md(d'_2)=\md(d_2)$. Hence, by the induction hypothesis, $\mdl{\cls{mFL_e}} \p(x)^\ast \le\x^\ast$ and $\mdl{\cls{mFL_e}} (\prod(\Ga',\x))^\ast \le(\sum\De)^\ast$. Since $\x$ is a sentence, the equations defining $\cls{mFL_e}$ yield also $\mdl{\cls{mFL_e}} (\Exi{x}\p(x))^\ast \le\x^\ast$. So $\mdl{\cls{mFL_e}} (\prod(\Pi,\Exi{x}\p(x)))^\ast  \le(\sum\De)^\ast$. 
\end{proof}

The proof-theoretic strategy described above extends easily to varieties of $\cls{FL_e}$-algebras axiomatized relative to $\cls{FL_e}$ by equations of a certain simple form. Given a variable $x$, let $x^0:=\ut$ and $x^{k+1}:=x\pd x^k$, for each $k\in\N$, and given a multiset $\Pi$ and $k\in\N$, let $\Pi^k$ denote the multiset union of $k$ copies of $\Pi$.  Now let $S$ be the set of equations $\{x\leq x^k\mid k\in\N\}\cup\{\zr\leq x\}$, and define sequent rules
\begin{align*}
r(x\leq x^k)=\vcenter{\infer[]{\Ga_1,\Pi,\Ga_2\seq\De}{\Ga_1,\Pi^k,\Ga_2\seq\De}}
\quad\text{and}\quad
r(\zr\leq x)=\vcenter{\infer[]{\Ga\seq\De}{\Ga\seq}}.
\end{align*}
Given any $S'\subseteq S$, denote by $\cls{FL_e}+S'$ the variety of $\cls{FL_e}$-algebras axiomatized relative to $\cls{FL_e}$ by the equations in $S'$, and by $\fFLe+r(S')$ the sequent calculus $\fFLe$ extended with the rules $r(\ep)$ for each equation $\ep$ in $S'$. Then for any sequent $\Ga\seq\De$ containing formulas from $\ofml$ (see,~e.g.,~\cite{OK85,Kom86}), 
\begin{align*}
\textstyle\der{\fFLe+r(S')}\Ga\seq\De \quad\Longleftrightarrow\quad\: \fosc{\cls{FL_{e}}+S'}\prod\Ga\le\sum\De.
\end{align*}
Moreover, the additional cases required to adapt the proof of Lemma~\ref{l:interpolation} to $\fFLe+r(S')$ are straightforward, since each application of a rule $r(\ep)$ for $\ep\in S'$ has just one premise. Hence, following the proof of Theorem~\ref{t:FLecomplete} yields the following more general result.

\begin{theorem}
For any $S'\subseteq S$ and set $T\cup\{\f\eq\p\}$ of $\ofml(\lang)$-equations,
\begin{align*}
T\fosc{\cls{FL_e}+S'}\f\eq\p \quad\Longleftrightarrow\quad T^\ast\mdl{\cls{mFL_e}+S'}\f^\ast \eq\p^\ast.
\end{align*}
\end{theorem}

\noindent
In particular, we obtain new completeness proofs for the axiomatizations of the one-variable fragments of the first-order extensions of $\lgc{FL_{ew}}$, $\lgc{FL_{ec}}$, and $\lgc{FL_{ewc}}$ (intuitionistic logic).


\section{Concluding remarks}\label{s:concluding}

Let us conclude this paper by mentioning some interesting directions for further research. The most general challenge for a class $\K$ of $\lang$-lattices may be stated as follows: provide a (natural) axiomatization of the equational consequence relation $\fosc{\K}$, or, equivalently, in algebraic terms, provide a  (natural) axiomatization of the generalized quasivariety generated by the class of all $\tuple{\alg{A},W}$-functional m-$\lang$-lattices where $\alg{A}\in\K$ and $W$ is any set. In this paper, we have shown that when $\K$ is a variety of $\lang$-lattices that has the superamalgamation property, the required generalized quasivariety is the variety $\mK$ of m-$\lang$-lattices  (Corollary~\ref{c:completeness}), axiomatized relative to $\K$ by a set of axioms familiar from modal logic. However, if $\K$ lacks the superamalgamation property or is not a variety, further axioms may be required.

One potential generalization is to consider varieties of $\lang$-lattices that have the weaker ``super generalized amalgamation property'', which corresponds for substructural logics (even those without exchange) to the Craig interpolation property~\cite{GJKO07}. In particular, such a result would yield an axiomatization for the one-variable fragment of the first-order version of the full Lambek Calculus $\lgc{FL}$, although we conjecture that completeness would hold only for valid equations and not consequences. Alternatively, such a generalization might be established proof-theoretically for first-order versions of substructural logics like $\lgc{FL}$ that have a cut-free sequent calculus, by lifting the proof-theoretic strategy presented in Section~\ref{s:prooftheory} to sequents based on sequences of formulas.

A further interesting line of inquiry concerns the case where $\K$ consists of the totally ordered members of a variety of $\lang$-lattices, and hence forms a positive universal class. First, let $\V$ be any variety of {\em semilinear} $\lgc{FL_e}$-algebras: algebras that are isomorphic to a subdirect product of totally ordered $\lgc{FL_e}$-algebras. It is not hard to show that in this case, $\fosc{\V}\Exi{x}\f\pd\Exi{x}\f \eq \Exi{x}(\f\pd \f)$. However, if $\textbf{\L}_3\in\V$ (e.g., if $\V$ is $\cls{MV}$ or the variety of all semilinear $\lgc{FL_e}$-algebras), then (as proved in Example~\ref{e:monadicvarieties}), $\nmdl{\cls{mV}}\di x\pd\di x \eq \di(x\pd x)$, so $\mV$ does not correspond to the one-variable fragment of the first-order logic based on $\V$.    

Now let $\cls{V_{to}}$ be the class of totally ordered members of $\V$. Then not only $\fosc{\cls{V_{to}}}\Exi{x}\f\pd\Exi{x}\f \eq \Exi{x}(\f\pd \f)$, but also $\fosc{\cls{V_{to}}}\All{x}(\f\jn\p)\eq\All{x}\f\jn\p$, where $x$ does not occur in $\p$. Although a general approach to obtaining axiomatizations of the one-variable fragments of the first-order logics based on $\V$ and $\cls{V_{to}}$ is lacking, success for specific cases indicate a possible way forward. Most notably, the one-variable fragment of first-order {\L}ukasiewicz logic can be defined over the class $\cls{MV_{to}}$ of totally ordered MV-algebras and corresponds to the  variety of monadic MV-algebras, defined relative to $\cls{mMV}$ by $\di x\pd\di x \eq \di(x\pd x)$ and $\bo(\bo x\jn y)\eq\bo x\jn\bo y$~\cite{Rut59}. Interestingly, a proof of this latter result is given in~\cite{CCVR20} using the fact that  $\cls{MV_{to}}$ has the amalgamation property (see also~\cite{MT20,Tuy21} for related results), suggesting that the approach developed in this paper might  be adapted to one-variable fragments of first-order logics based on classes of totally ordered algebras that have the amalgamation property.


\bibliographystyle{asl}


\end{document}